\documentclass[10pt,a4paper]{article} 
\usepackage{amsmath,amssymb,amsthm,amsfonts,amscd,euscript,verbatim, t1enc, newlfont, graphicx, cancel}
\usepackage{hyperref}
\usepackage[all]{xy}
\usepackage{stackengine,scalerel}
\providecommand{\keywords}[1]{\textbf{\textit{Key words and phrases }} #1}

\theoremstyle{definition}
\newtheorem{theo}{Theorem}[subsection]

\newtheorem{pr}[theo]{Proposition}

 \newtheorem{lem}[theo]{Lemma}

\theoremstyle{remark}
\newtheorem{rema}[theo]{Remark}

\theoremstyle{definition}
\newtheorem{defi}[theo]{Definition}

\numberwithin{equation}{subsection}

\newcommand\cu{\underline{C}}

\newcommand\au{\underline{A}}

\newcommand\obj{\operatorname{Obj}}
\newcommand\mo{\operatorname{Mor}}
\newcommand\id{\operatorname{id}}

\DeclareMathOperator\imm{\operatorname{Im}}

\newcommand\hw{{\underline{Hw}}}

\newcommand\lo{\mathcal{LO}}
\newcommand\ro{\mathcal{RO}}

\newcommand\jo{\mathcal{J}}

\newcommand\z{{\mathbb{Z}}}

\newcommand\ns{\{0\}}
\newcommand\ab{\operatorname{Ab}}

\newcommand\cp{\mathcal{P}}
\newcommand\perpp{{}^{\perp}}
\newcommand\opp{^{op}}

\newcommand\wsp{w^{sph}}

\newcommand\shtop{SH}

\newcommand\hsing{H^{sing}}
\newcommand\hsingc{H_{sing}}

\begin{document}

\title{Killing weights from  the perspective of  $t$-structures} 
 \author{Mikhail V.Bondarko\thanks{The  work is supported by Russian Science Foundation grant no. 20-41-04401.\break  
St. Petersburg State University,  7/9 Universitetskaya nab., St. Petersburg, 199034 Russia; St. Petersburg Department of Steklov Math. Institute, Fontanka 27, St. Petersburg, 191023, Russia; m.bondarko@spbu.ru.}, Sergei V. Vostokov\thanks{St. Petersburg State University,  7/9 Universitetskaya nab., St. Petersburg, 199034 Russia,  s.vostokov@spbu.ru}\\ To  Aleksei Nikolaevich Parshin, on the occasion of his 80th Birthday.}
 \maketitle
\begin{abstract}
This paper is devoted to morphisms killing weights in a range (as defined by the first author) and to objects without these weights (as essentially defined by J. Wildeshaus) in a triangulated category endowed with a weight structure $w$. We describe several new criteria for morphisms and objects to satisfy these conditions. In some of them we use virtual $t$-truncations and a $t$-structure adjacent to $w$. 
 In the case where the latter exists we prove that a morphism kills weights $m,\dots,n$ if and only if it factors through an object without these weights; we also 
  construct new families of torsion theories and projective and injective classes. 
  As a consequence, we obtain some  "weakly functorial decompositions" of spectra (in the stable homotopy category $SH$) and a new description of those morphisms that act trivially on  
  singular cohomology $\hsingc^0(-,\Gamma)$ with coefficients in every abelian group $\Gamma$. 
\end{abstract}



 \section*{Introduction}
 
 This note is a certain complement to \cite{bkwn} whose central notion was that of  {\it morphism killing weights} $m,\dots,n$ in a triangulated category $\cu$ endowed with a {\it weight structure} $w$; here $m\le n\in \z$. One says that an  object $M$ of $\cu$ is {\it without these weights} if $\id_{M}$ kills  weights $m,\dots,n$, and it was proved in ibid. that this version of this notion is closely related to the original definition of J. Wildeshaus. 
  One may consider these notions in the important case $\cu=SH$ and $w=\wsp$ (this is the {\it spherical weight structure} on the stable homotopy category); in this case $M$ is without weights $m,\dots,n$ whenever its singular homology $\hsing_{*}(M,\z)$  vanishes in these degrees and $\hsing_{m-1}(M,\z)$ is a free abelian group. 

In the current paper we establish some new criteria for killing weights $m,\dots,n$ for a $\cu$-morphism (in  Theorem \ref{tkw}) and for the absence of these weights in $M\in \obj \cu$ (see Theorem \ref{tww}). In  particular, if 
 $w=\wsp$ 
  then $X$ is without weights $m,\dots,n$ if and only if it possesses a cellular tower with  $X^{(m-1)}=X^{(n)}$ (these are the corresponding skeleta of $X$; see  \cite[\S6.3]{marg}). In contrast to \cite{bkwn}, we mention weight filtrations and virtual $t$-truncations of cohomology in some of our criteria.
 
 Moreover, we study the case where $\cu$ is also endowed with a $t$-structure $t$ {\it adjacent} to $w$; note that $t$ of this sort exists for $w=\wsp$, and 
  if $t$ exists then virtual $t$-truncations of representable functors are represented by the corresponding $t$-truncations. 
 In this case a morphism kills weights $m,\dots,n$ if and only if it factors through an object without these weights; see 
  Theorem \ref{tptp}. Moreover, this theorem
     gives 
   certain torsion theories and projective and injective classes in $\cu$ (these notions were defined in \cite{iyayo} and \cite{christ}, respectively); 
    this yields some more criteria for killing weights and the   absence of weights.  
    
    In Theorem \ref{thegcomp} we discuss the application of these results to the case where $w$ is {\it purely compactly  generated}; this includes the case $w=\wsp$. In particular,  for an $SH$-morphism $g$ we have $\hsingc^0(g,\Gamma)=0$ for every abelian group $\Gamma$ (that is, 
     $g$ annihilates singular cohomology with coefficients in $\Gamma$)  if and only if $g$ factors through a spectrum $X$ such that all $\hsingc^0(X,\Gamma)$ vanish; see Remark \ref{rsh}(1). Furthermore, this statement naturally extends to the equivariant stable  homotopy category  $SH(G)$; see Remark \ref{rsh}(2).


The authors are deeply grateful to the referee for useful comments to the text.  The first author 
 also thanks  
  the Max Planck Institut für Mathematik for the hospitality  during 
 finishing this version of the paper.

\section{Preliminaries}\label{sold}
In this section we recall several relevant definitions and statements. We do not prove anything new.

\subsection{Some categorical notation }\label{snotata} 

\begin{itemize}

\item Let $C$ be a category and  $X,Y\in\obj C$. Then  we will write
$C(X,Y)$ for  the set of morphisms from $X$ to $Y$ in $C$.


\item 
We say that 
 $X$ is a {\it retract} of $Y$ if $\id_X$ can be 
 factored through $Y$. 

\item A 
 subcategory $\cp$ of an additive category $C$ 
is said to be {\it retraction-closed} in $C$ if it contains all retracts of its objects in $C$.


\item The symbol $\cu$ below will always denote some triangulated category; usually it is  endowed with a weight structure $w$.
Moreover, $w$ will always denote a weight structure on $\cu$, and $t$ is a $t$-structure on $\cu$ (see Definition \ref{dhopo}(\ref{itw},\ref{itt}) below). 

$\au$ will   denote some abelian category.


\item For any  $A,B,C \in \obj\cu$ we  say that $C$ is an {\it extension} of $B$ by $A$ if there exists a distinguished triangle $A \to C \to B \to A[1]$.


\item For $X,Y\in \obj \cu$ we 
 write $X\perp Y$ if $\cu(X,Y)=\ns$. If $D$ and $E$ are classes of objects or subcategories of $\cu$ then 
we will write $D\perp E$ if $X\perp Y$ for all $X\in D,\
Y\in E$. Moreover,  we  write $D^\perp$ for the class
$$\{Y\in \obj \cu:\ X\perp Y\ \forall X\in D\};$$
dually, ${}^\perp{}D$ 
is the class $
 \{Y\in \obj \cu:\ Y\perp X\ \forall X\in D\}$.



\item For an object $M$ of $\cu$ we will write $H_M$ for the functor $\cu(-,M)$.

\end{itemize}

\subsection{On various torsion theories and 
 projective classes}\label{ssws}

Our central definition is the following one.

\begin{defi}\label{dhop}
 A couple $s$ of classes $\lo,\ro\subset\obj \cu$ 
is said to be a {\it torsion theory} (on $\cu$) if $\lo^{\perp}=\ro$,  $\lo={}^{\perp}\ro$, and 
for any $M\in\obj \cu$ there
exists a distinguished triangle
\begin{equation}\label{swd}
L_sM\stackrel{a_M}{\longrightarrow} M\stackrel{n_M}{\longrightarrow} R_sM
{\to} L_sM[1]\end{equation} 
such that $L_sM\in \lo $ and $ R_sM\in \ro$. We will call any triangle of this form an {\it $s$-decomposition} of $M$. 
\end{defi}

We also need a collection of related definitions.

\begin{defi}\label{dhopo}
\begin{enumerate}
Let $s$ be a torsion theory and $n\in \z$.

\item\label{itw}
We will say that $s$ is {\it weighted} if $\lo\subset \lo[1]$.

In this case we  call the couple $w=(\lo, \ro[-1])$ a {\it weight structure} and say  that $s$ is associated with $w$.  
 We write $\cu_{w\le 0}$ and  $\cu_{w\ge 0}$ for $\lo$ and $ \ro[-1]$, respectively; furthermore, we set $\cu_{w\le n}=\cu_{w\le 0}[n]$ and $\cu_{w\ge n}=\cu_{w\ge 0}[n]$.

\item\label{itt}
 If $\lo\subset \lo[-1]$ then we  call the couple $t=(\cu_{t\le 0}=\ro[1],\ \cu_{t\ge 0}=\lo)$  a {\it $t$-structure} and say  that $s$ is associated with $t$. 
We set  $\cu_{t\le n}=\cu_{t\le 0}[n]$ and $\cu_{t\ge n}=\cu_{t\ge 0}[n]$. Moreover, if $m\le n$ then we set $\cu[m,n]=\cu_{t\le n}\cap \cu_{t\ge m}$.

\item\label{itwadj} We say that a weight structure $w$ and a $t$-structure $t$ (as above) are {\it adjacent} if $\cu_{w\ge 0}=\cu_{t\ge 0}$.  

\item\label{ittg}
We will say (following \cite[Definition 3.1]{postov}) that $s$ is {\it generated by $\cp\subset \obj \cu$} if $\cp^\perp=\ro$. 

\end{enumerate}
\end{defi}

\begin{rema}\label{rtt}
\begin{enumerate}
\item\label{irtt2}
Our definition of torsion theory actually follows \cite[Definition 3.2]{postov} (where torsion theories were called {\it complete Hom-orthogonal pairs})  and is somewhat different
from Definition 2.2 of \cite{iyayo}, from which our term comes from.  However,  these two definitions are  well-known to be equivalent; see  Proposition 2.4(2,9) and Remark 2.5(1) of \cite{bvt}. 

\item\label{irttwt} Our definitions of weight and $t$-structures are equivalent to the ones given in \cite{bkwn} and \cite{bpws};  see Proposition 3.2 of \cite{bvt}. Note however that in 
  \cite{bpws} the so-called homological conventions for  weight and $t$-structures were used, and they differ from the original ("cohomological") conventions originally introduced in \cite{bbd} and \cite{bws} (cf. Remarks 1.2.4(4) and 2.2.3(3) of \cite{bpws}).\footnote{Also recall that D. Pauksztello has introduced weight structures independently 
 (in \cite{konk}); he called them co-t-structures.} 

\item\label{irttadj}
 Dually to Definition  \ref{dhopo}(\ref{itwadj}), one can also consider the case where  $\cu_{w\le 0}=\cu_{t\le 0}$; cf. Definition 2.2.2(6) of \cite{bpws} where both of the versions are defined.  However, we will not need this 
 dual 
  setting in the current paper. 
\end{enumerate}
\end{rema}


Let us recall some more definitions.

\begin{defi}\label{dinj}
 Assume $\cp \subset \obj \cu$ and $\jo\subset \mo(\cu)$.

1. We will say that a $\cu$-morphism $h$ is {\it $\cp$-null} (resp. {\it $\cp$-conull})  whenever for all $M\in \cp$  we have $\cu(M,h)=0$ (resp. $\cu(h,M)=0$). 



2. The couple $(\cp,\jo)$ is said to be a {\it projective class} whenever the following conditions are fulfilled:

(i) $\jo$ is the class of all $\cp$-null morphisms;

	(ii) If $M\in \obj\cu$  then the functor $\cu(M,-)$ annihilates all elements of $\jo$ if and only if $M\in \cp$.

(iii) For any $M\in\obj \cu$ there
exists a distinguished triangle
\begin{equation}\label{iwd}
 PM\stackrel{i_M}{\to} M\stackrel{j_M}{\to} 
 IM\to 
 PM[1]\end{equation} 
such that $P_M\in \cp$ and $j_M\in \jo$.

3. We say that $(\cp,\jo)$ is an {\it injective class} if it becomes a projective class in the category $\cu\opp$.
\end{defi}


\subsection{On weight range and virtual $t$-truncations}\label{srange}

Let us recall some properties of truncations with respect to weight and $t$-structures. Recall that $w$ (resp. $t$) always denotes a  weight structure (resp.  $t$-structure) on $\cu$.

\begin{pr}\label{prtst}
Let $M,X,Y\in \obj \cu$. 

\begin{enumerate}
\item\label{itcan}
The decomposition triangle 
(\ref{swd}) is essentially canonically and functorially determined by $M$ if $s$ is associated with $t$ (see Definition \ref{dhopo}(\ref{itt})). 

Consequently, if $n\in \z$ and $N=M[n]$ then
  both $t_{\le n-1}N=R_sM[n]$ and $t_{\ge n}N=L_sM[n]$ are functorially determined by $N$. 


\item\label{it2s} Moreover, 
 if $m\le n$ 
 then $t_{\le n}(\cu_{t\ge m})= \cu[m,n]$.

\item\label{iwst}
 Assume that $s$ is associated with $w$ (see Definition \ref{dhopo}(\ref{itw})).  For any  $n\in \z$ and $N=M[n]$ we set  $w_{\le n}N=L_sM[n]$ and $w_{\ge n+1}N=R_sM[n]$. 
 
 Then for any $m\le l\in \z$ any $g\in \cu(X,Y)$ can be extended 
to a 
morphism of the corresponding distinguished triangles (cf. Remark \ref{rwd}(1) below):
 \begin{equation}\label{ecompl} \begin{CD} w_{\le m} X@>{}>>
X@>{}>> w_{\ge m+1}X\\
@VV{}V@VV{g}V@ VV{}V \\
w_{\le l} Y@>{}>>
Y@>{}>> w_{\ge l+1}Y \end{CD}
\end{equation}

Moreover, if $m<l$ then this extension is unique provided that the rows are fixed.
\end{enumerate}
\end{pr}
\begin{proof}
\ref{itcan}. Well known; see Proposition 1.3.3 of  
  \cite{bbd} (yet pay attention to Remark \ref{rtt}(\ref{irttwt})). 

\ref{it2s}.  Proposition 1.3.5 of 
 ibid. immediately implies that  $t_{\le n}(\cu_{t\ge m})$ lies in $ \cu[m,n]$. 
    Next,  
 Proposition 1.3.3 of ibid. easily yields $ \cu[m,n]\subset t_{\le n}(\cu_{t\ge m})$. 

\ref{iwst}. Easy; see Lemma 1.5.1(1,2) of \cite{bws}.
\end{proof}

\begin{rema}\label{rwd}
1. The upper row in (\ref{ecompl}) is called an {\it $m$-weight decomposition} of $X$. It is easily seen that this triangle is not canonically determined by $(w,X)$. Yet it will be convenient for us to use this notation below. Moreover,  when we will write arrows of the type $w_{\le m}X\to X$ or $X\to w_{\ge m+1}X$ we will always assume that they come from some $m$-weight decomposition of $X$. 

2.  Proposition \ref{prtst}(\ref{iwst}) says that $m$-weight decompositions are "weakly functorial", that is, a morphism $g$ between objects extends to a morphism between any choices of their $m$-weight decompositions. According to Proposition 2.4(8) of \cite{bvt}, this property is fulfilled for  $s$-decomposition triangles (see (\ref{swd})) as well. 

\end{rema}

Now we 
 pass to weight filtrations and virtual $t$-truncations. 

\begin{defi}\label{dwfil}
Let $H$ be a 
contravariant functor from $\cu$ into $\au$ (where $\au$ is an abelian category), $m\le n\in \z$,  and $M\in \obj \cu$.

1. We define the weight filtration for $H(M)$  as $$W^m(H)(M)=\imm(H(w_{\ge m}M)\to H(M));$$ 
here we take an arbitrary choice of  $w_{\ge m}M$ and use the convention described in Remark \ref{rwd}. 

2.  We define the 
 correspondence (cf. Proposition \ref{pwfil}(\ref{iwfil1}) below)   $\tau_{\le m }(H)$ 
  as  $$M\mapsto\imm (H(w_{\le m+1}M)\to H(w_{\le m}M)) ;$$ here we take arbitrary choices of  
$w_{\le m}M$ and $ w_{\le m+1}M$, and take $X=Y=M$ and $g=\id_M$ in (\ref{ecompl}).

3. 
If $H$ is cohomological (that is, it converts distinguished triangles into long exact sequences), we will say that it is {\it of weight range} $\ge m$ if it annihilates $\cu_{w\le m-1}$. 
 $H$ is   {\it of weight range $[m,n]$} if it  annihilates $\cu_{w\ge n+1}$ as well.
\end{defi}

\begin{pr}\label{pwfil}
In the notation of the previous definition the  following statements are valid.

\begin{enumerate}

\item\label{iwfil1} $W^mH(M)$ and  $\tau_{\le m}(H)$ are $\cu$-functorial in $M$ (for any $m$; in particular, they essentially do not depend on the choices of the corresponding  weight decompositions of $M$).

\item\label{iwfil2} If $H$ is cohomological then the functor  $\tau_{\le m}(H)$ also is.

\item\label{iwfil3} The functor $H_M=\cu(-,M)$ 
 is of weight range $\ge m$ if and only if $M\in \cu_{w\ge m}$.

\item\label{iwfil4} If $H$ is 
 of weight range $\ge m$  then  $\tau_{\le n}(H)$ is 
    of weight range $[m,n]$.

\item\label{iwfil5} If $H$ is of weight range $[m,n]$ then the morphism  $H(w_{\ge m}M)\to H(M)$ is surjective and the morphism  
$H(M)\to H(w_{\le n}M)$ is injective (here we take arbitrary choices of the corresponding weight decompositions of $M$ and apply $H$ to their connecting morphisms). 

\item\label{iwadjt}
 Assume that there exists a $t$-structure $t$ adjacent to $w$. Then  
 the functor  $\tau_{\le m}(H_M)$ 
 is represented by $t_{\le m}M$.


\end{enumerate}

\end{pr}
\begin{proof}
Assertions  \ref{iwfil1}, \ref{iwfil2}, and \ref{iwadjt} were established in \cite{bws}; see Propositions 2.1.2(2) and 
 2.5.1(I, III.2) and Theorem 4.4.2(8) of ibid. (and 
 mind the difference in notation).  



Assertion \ref{iwfil3} is an immediate consequence of our definitions. 

Assertions \ref{iwfil4} and  \ref{iwfil5} are given by Proposition 2.1.4(8,12) of \cite{bvtr}. 
\end{proof}

\section{On morphisms killing weights and objects without weights in a range}\label{skw}

 In this section 
 we recall the central definitions of \cite{bkwn}. We also prove some new criteria for  killing weights  (in terms of virtual $t$-truncations and weight filtrations; see Theorem \ref{tkw}) and for the absence of weights $m,\dots,n$ (in objects; see Theorem \ref{tww}).

\subsection{On morphisms killing weights: definitions and new criteria}\label{skwm}

\begin{defi}\label{dkw}
 Let $m\le n\in \z$.

 We say that a morphism $g\in \cu(M,N)$ in a weighted category {\it kills weights} $m,\dots,n$ whenever $g$ fulfils  the following equivalent conditions (see Proposition 2.1.1 of \cite{bkwn}).

1. There exists a choice of  $w_{\le n}M$ and $w_{\le m-1}N$ 
 along with the corresponding connecting morphisms (see Remark \ref{rwd}(1)) and a morphism $h$ making the following square commutative:
\begin{equation}\begin{CD} \label{ekw}
w_{\le n}M@>{x}>>M
\\
@VV{
h}V@VV{g}V\\
w_{\le m-1}N@>{}>>N 
\end{CD}\end{equation}

2. There exists a choice of  $w_{\le n}M$ and $w_{\ge m}N$ such that the corresponding composed morphism $w_{\le n}M\stackrel{x}{\to}  M\stackrel{g}{\to}N\stackrel{y}{\to}    w_{\ge m}N$ is zero. 

3. Any choice of the rows in 
the diagram
\begin{equation}\label{eccompl} \begin{CD} w_{\le n} M@>{x}>>
M@>{}>> w_{\ge n+1}M\\
@VV{h}V@VV{g}V@ VV{}V \\
w_{\le m-1} N@>{}>>
N@>{y}>> w_{\ge m}N \end{CD}\end{equation}
can be completed to the whole diagram.

We will write $\mo_{\cancel{[m,n]}}\cu$ for the class of $\cu$-morphisms that  kill weights $m,\dots,n$.
\end{defi}

Now we relate these conditions to weight range of functors (see Definition \ref{dwfil}).

\begin{theo}\label{tkw}

Adopt the notation of Definition \ref{dkw}.
Then  the following 
conditions are equivalent.

\begin{enumerate}
\item\label{iwf1} $g$ kills weights $m,\dots,n$.

\item\label{iwf2} $H(g)$ sends $W^{m}(H)(N)$ inside  $W^{n+1}(H)(M)$ for any contravariant functor $H:\cu\to \au$.

\item\label{iwf3} $H(g)$ sends $W^{m}(H_I)(N)$ inside  $W^{n+1}(H_I)(M)$ for 
all $I\in \cu_{w\ge m}$.

\item\label{iwf4} $H(g)=0$ if $H$ is an arbitrary cohomological functor  ($\cu\to \au$) of weight range $[m,n]$.

\item\label{iwf5} $H(g)=0$ for $H=\tau_{\le n }(H_I)$ whenever $I\in \cu_{w\ge m}$.

\item\label{iwf0} $H(g)=0$, where $H=\tau_{\le n }(H_{I_0})$ and $I_0$ is some  fixed choice of  $w_{\ge m}N$. 

\end{enumerate}

\end{theo}
\begin{proof}

Clearly, condition  \ref{iwf2} implies condition  \ref{iwf3}. Next,  \ref{iwf4} implies condition  \ref{iwf5} by Proposition \ref{pwfil}(\ref{iwfil3},\ref{iwfil4}), and the latter condition clearly implies condition  \ref{iwf0}. 

 Now assume  that $g$ kills weights $m,\dots,n$. Then we have a commutative diagram
\begin{equation}\begin{CD} \label{ekwn}
 M @>{}>>w_{\ge n+1}M
\\
@VV{
g}V@VV{j}V\\
N @>{y}>>w_{\ge m}N 
\end{CD}\end{equation}
(it does not matter here whether we fix some choices of the rows or not; see condition 3 in Definition \ref{dkw}). Applying $H$ to this diagram, we obtain condition  \ref{iwf2}.

Next we fix some choice of the rows of (\ref{ekwn}) and take $I=w_{\ge m}N$. 
 Assume that $g$ fulfils condition  \ref{iwf3}; then the morphism $d\circ g$  belongs to the image of $\cu(w_{\ge n+1}M,w_{\ge m}N)$  in $\cu(M,w_{\ge m}N)$. Thus there exists a choice of $j$ that makes (\ref{ekwn}); hence $g$ kills weights $m,\dots,n$ (see condition 1 in Definition \ref{dkw}).  

It remains to deduce condition  \ref{iwf4} from  \ref{iwf1}, and deduce the latter one from condition   \ref{iwf0}.

Assume  that $g$ kills weights $m,\dots,n$. If $H$ is a (cohomological) functor of weight range $[m,n]$ then the morphism $H(y):H(w_{\ge m}N)\to H(N)$ is surjective and the morphism  
$H(x):H(M)\to H(w_{\le n}M)$ is injective (for any choices of the corresponding weight decompositions); see Proposition \ref{pwfil}(\ref{iwfil5}).  Since the composed morphism $a:w_{\le n}M\to w_{\ge m}N$ is zero (see 
 condition 3 in Definition \ref{dkw}; thus $H(a)=0$ as well), we obtain  condition  \ref{iwf4}. 

Now assume that  condition   \ref{iwf0} is fulfilled. Consider the element $r$ of the group $ \tau_{\le n }(H_{I_0})(N)=\imm(\cu(w_{\le n+1}N,I_0)\to \cu(w_{\le n}N,I_0)) $ 
  obtained by 
composing the corresponding connecting morphisms (recall that $I_0=w_{\ge m}N$). Since $r$ vanishes in  $\tau_{\le n }(H_{I_0})(M)\subset \cu(w_{\le n}M,I_0)$, the composed morphism $a\in \cu(w_{\le n}M, w_{\ge m}N)$ is zero. Thus  we obtain condition  \ref{iwf1}; see condition 2 in Definition \ref{dkw}.
\end{proof}

\begin{rema}\label{rkwz}
 In \cite{bkwn} much attention was paid to morphisms {\it killing weight $m$}, 
  that is, to the case $n=m$ of Definition \ref{dkw} (cf. Remark \ref{rsh}(1)  below).  
   Recall that functors of weight range $[0,0]$ are the (cohomological) {\it pure} ones in the sense of \cite[Definition 2.4.1]{bkwn}, and they can be expressed in terms of {\it weight complexes}; see \S1.3 and Theorem 2.4.2(3) of ibid. This yields a relation between weight complexes and killing weight $m$; see Theorem 2.3.1(1) of ibid. 
   
 
 On the other hand, virtual $t$-truncations and weight filtrations are not mentioned in ibid.
\end{rema} 

\subsection{On objects without weights $m,\dots, n$}\label{skww}

Now we pass to a class of objects that is important for this paper.

\begin{defi}\label{dww}
We say that an object $M$ of a weighted category $\cu$ is {\it without weights} $m,\dots, n$ (for $m\le n\in \z$) whenever $\id_M$ kills these weights.

We write $\cu_{w\notin{[m,n]}}$ for the class of objects of this sort. 
\end{defi}

\begin{theo}\label{tww}
 I. Under the assumptions  of Definition \ref{dww} the following 
conditions are equivalent.
\begin{enumerate}
\item\label{iwf6} $M$ is without weights $m,\dots,n$.

\item\label{iwf6n} $M$  is a retract of some $\widetilde M\in \obj \cu$ 
 that is  an extension of an element of $\cu_{w\ge n+1}$ by  an element of $\cu_{w\le m-1}$.

\item\label{iwf7} $H(M)=0$ 
 if $H$ of weight range $[m,n]$.

\item\label{iwf8} $H(M)=\ns$  for $H=\tau_{\le n }(H_I)$ 
 whenever $I\in \cu_{w\ge m}$. 

\item\label{iwf9} $H(M)=\ns$, where $H=\tau_{\le n }(H_{I_0})$ and $I_0$ is a  fixed choice of  $w_{\ge m}M$. 

\end{enumerate}

II. Moreover, if $\cu$ is {\it idempotent complete}, that is, if all idempotent endomorphisms split in it (cf. Proposition 1.6.8 of \cite{neebook}), then the conditions above are also equivalent to the following ones.

1. $M$ is  an extension of an element of $\cu_{w\ge n+1}$ by  an element of $\cu_{w\le m-1}$ itself.

2. There exists isomorphic choices of the weight truncations $w_{\le m-1}M$ and  $w_{\le n}M$.

3. There exists a choice of $w_{\le m-1}M$ and  $w_{\le n}M$ such that the corresponding connecting morphism defined as in Definition \ref{dwfil}(2) is an isomorphism.

\end{theo}
\begin{proof}
I. It is easily seen that Theorem \ref{tkw} implies the equivalence of conditions I.\ref{iwf6} and I.\ref{iwf7}--\ref{iwf9}. 
Moreover, this theorem (see condition I.\ref{iwf4} in it) gives the implication  I.\ref{iwf6n} $\implies$  I.\ref{iwf6}.

Now assume that $M$ is without weights $m,\dots,n$. Then combining Theorem 2.2.1(6,10) with Proposition 3.1.2(2) of \cite{bkwn} we obtain the existence of a triangulated category $\cu'\supset \cu$ 
 along with a $\cu'$-distinguished triangle 
 \begin{equation}\label{ewdp}
 L'M\to M\to R'M\to L'M[1]
 \end{equation}
  such that $L'M$ is a retract of an element $\widetilde{LM}$ of $\cu_{w\le m-1}$ and $R'M$ is a retract of 
  $\widetilde{RM}\in \cu_{w\ge n+1}$. Since $\cu'$ is triangulated, there exist objects $L''M$ and $R''M$  of the category $\cu'$ such that $L'M\bigoplus L''M\cong \widetilde{LM}$ and $R'M\bigoplus R''M\cong \widetilde{RM}$. Thus we can add \ref{ewdp} with (say) a split distinguished triangle $$L''M\to M''=L''M\bigoplus R''M \to R''M\to L''M[1]$$ to obtain the triangle 
$\widetilde{LM}\to M\bigoplus M''\to  \widetilde{RM}\to \widetilde{LM}[1]$. Since $\widetilde{LM}$ and $\widetilde{RM}$ are objects of $\cu$, $M\bigoplus M''$ is $\cu'$-isomorphic to an object of $\cu$. Thus one can take $\widetilde{M}\cong M\bigoplus M''$ and obtain that  condition I.\ref{iwf6} implies condition I.\ref{iwf6n}.

II. Theorem 2.2.1(9,10) of \cite{bkwn} states that 
conditions I.\ref{iwf6} and II.1 are equivalent whenever $\cu$ is {\it weight-Karoubian}, that is, if  the category $\hw=\cu_{w\le 0}\cap \cu_{w\ge 0}$ is idempotent complete. 
Now, this assumption is easily seen to follow from the idempotent completeness of $\cu$; see Proposition 1.2.4(7) of ibid. 

Lastly, the equivalence of conditions II.1--3 is very easy, and we leave it as an exercise to the reader.

\end{proof}

\begin{rema}\label{rwild}
1. Originally, objects without weights were defined by J. Wildeshaus. His   
 \cite[Definition 1.10]{wild} coincides with condition II.1 of Theorem \ref{tww}. Thus his definition is equivalent to ours  whenever $\cu$ is {\it weight-Karoubian} (see the proof of Theorem \ref{tww}(II)). Yet this equivalence fails in general; see \S3.3 of \cite{bkwn}.

 2. In the case where there exists a $t$-structure $t$ adjacent to $w$ (on $\cu$) some more descriptions of $\mo_{\cancel{[m,n]}}\cu$ and $\cu_{w\notin{[m,n]}}$ are provided by Theorem \ref{tptp} below.
 \end{rema}

Lastly we make a simple nice observation.

\begin{lem}\label{lfact}
Assume that a $\cu$-morphism $h$  factors through 
 some $M\in \cu_{w\notin{[m,n]}}$ 
  (for $m\le n\in \z$). Then $h$ kills weights  $m,\dots, n$.
\end{lem}
\begin{proof}
Our assumption means that $h=a\circ \id_M\circ b$ for some $\cu$-morphisms $a$ and $b$.  Applying Theorem 2.2.1(3) of \cite{bkwn} we obtain that  $h$ kills weights  $m,\dots, n$ indeed.
\end{proof}

\begin{rema}\label{req}
In Theorem \ref{tptp}(3) below we prove the converse implication in the case where a $t$-structure adjacent to $w$ exists. The authors suspect that this equivalence fails in general. In particular, it would be interesting to take $\cu$ to be the subcategory of finite spectra in $SH$; cf. Theorem \ref{thegcomp} below and Theorem 4.2.1(1) of \cite{bwcp}.

\end{rema}

\section{On the relation to the adjacent $t$-structure}\label{skwadj}

In \S\ref{sntt} we construct some "new" torsion theories and injective classes on $\cu$ whenever it is endowed with adjacent (weight and $t$-structure) $w$ and $t$.

In \S\ref{sex} we consider the so-called purely compactly generated examples of this setting. In particular, we discuss the application of our results to the case $\cu=SH$.

\subsection{Some new torsion theories and injective classes}\label{sntt}


\begin{theo}\label{tptp}
Assume that $\cu$ is endowed with a weight structure $w$ and an adjacent $t$-structure $t$, and $m\le n\in \z$.

1. Then the couple $s=(\cu_{w\notin{[m,n]}},\cu[m,n])$ is a torsion theory.

2.  The couple $(\cu_{w\notin{[m,n]}},\jo_{[m,n]})$ is a projective class, where $\jo_{[m,n]}$ is the class of those morphisms that factor through $\cu[m,n]$ (cf. Lemma \ref{lfact}). 

3.  The couple $(\cu[m,n],\mo_{\cancel{[m,n]}}\cu)$ is an  injective class, and $\mo_{\cancel{[m,n]}}\cu$ consists of those morphisms that factor through $\cu_{w\notin{[m,n]}}$.


\end{theo}
\begin{proof}
1. According to 
 Proposition 2.4(9) of \cite{bvt} it suffices to verify that the classes $\cu_{w\notin{[m,n]}} $ and $\cu[m,n]$ are retraction-closed in $\cu$, 
$\cu_{w\notin{[m,n]}} \perp \cu[m,n]$, and for any $M\in \obj \cu$ there exists an $s$-decomposition (\ref{swd}).

The class $\cu[m,n]$ is retraction-closed in $\cu$ since it equals the intersection of the  classes $\cu_{t\le n}$ and $\cu_{t\ge m}$, whereas the 
 equalities $\lo_{s^t}^{\perp}=\ro_{s^t}$ and  $\lo_{s^t}={}^{\perp}\ro_{s^t}$ (for the torsion theory $s^t$ 
  associated with $t$) imply that these two classes are retraction-closed in $\cu$. 
 Moreover, 
 it follows from  Lemma \ref{lfact} 
   that $\cu_{w\notin{[m,n]}} $  is retraction-closed in $\cu$ as well. 

Next, the aforementioned orthogonality conditions in Definition \ref{dhop} (applied to the torsion theories associated with $w$ and $t$) along with condition I.\ref{iwf7} in Theorem \ref{tww} easily yield $\cu_{w\notin{[m,n]}} \perp \cu[m,n]$. 

It remains to verify the existence of an $s$-decomposition for an arbitrary $M\in \obj \cu$.

We fix some $w_{\ge m}M$, denote   $ t_{\le n}(w_{\ge m}M)$ by $RM$, and 
 complete the corresponding composed morphism $h\in \cu(M,RM)$ to a triangle $LM\to M\to RM\to LM[1]$. Then $LM$ is an extension of $t_{\ge n+1}(w_{\ge m}M)$ by $w_{\le m-1}M$ (by the octahedron axiom of triangulated categories). Since $t_{\ge n+1}(w_{\ge m}M)\in \cu_{w\ge n+1}$  (by the definition of adjacent structures), $LM$ is without weights $m,\dots, n$; see 
  condition I.\ref{iwf6n} in Theorem \ref{tww}. Lastly, since  $w_{\ge m}M\in \cu_{t\ge m}$, we have $RM\in \cu[m,n]$ by Proposition \ref{prtst}(\ref{it2s}). 

2. 
 The statement immediately follows from (the simple) Lemma \ref{lbvt} below.
 
 3. Applying  Lemma \ref{lbvt} to the category $\cu\opp$ we obtain that  $(\cu[m,n],\jo'[m,n])$ is 
  an injective class, where $\jo'[m,n]$ is the class of $\cu[m,n]$-conull morphisms. Moreover, $\jo'[m,n]$  coincides with the class of morphisms that factor through $\cu_{w\notin{[m,n]}}$.

It remains to verify that  $\jo'[m,n]=\mo_{\cancel{[m,n]}}\cu$. The latter statement easily follows from Propositions \ref{pwfil}(\ref{iwadjt}) and \ref{prtst}(\ref{it2s}); 
   see condition \ref{iwf5} in Theorem \ref{tkw}. 
\end{proof}



\begin{lem}\label{lbvt}
Let $(\lo,\ro)$ be a torsion theory. Then $(\lo,\jo)$ is a projective class, where $\jo$ is the class of $\lo$-null morphisms. Moreover, $\jo$ coincides with the class of those morphisms that factor through $\ro$.
\end{lem}
\begin{proof} This is Proposition 5.2 of \cite{bvt}.
\end{proof}

\begin{rema}
1. In particular,  $\mo_{\cancel{[0,0]}}\cu$ is the class of $\cu[0,0]$-conull morphisms. This statement is easily seen to generalize Corollary 4.1.6(3) of \cite{bkwn}; see Theorem \ref{thegcomp}(\ref{itnpw},\ref{itnpt}) below. Recall here that 
 in some 
   papers of the  first author the class $\cu[0,0]$ was denoted by $\cu_{t=0}$ (cf. Definition 4.3.1 of \cite{bwcp}).

2. There exists one more description of the 
 class $\jo_{[m,n]}$ 
  in Theorem \ref{tptp}. Indeed, Lemma \ref{lbvt} implies that 
   the couples $(\cu_{w\ge n+1},\jo_1)$ and  $(\cu_{w\le m-1},\jo_2)$ are projective classes;
   here $\jo_1$ is the class of morphisms that factor through (elements of) $\cu_{w\ge n+1}\perpp=\cu_{t\ge n+1}\perpp=\cu_{t\le n}$ and $\jo_2$ consists of those morphisms that factor through $\cu_{w\le m-1}\perpp=\cu_{w\ge m}=\cu_{t\ge m}$. Thus Proposition 3.3 of \cite{christ} 
    yields that  $(\cu_{w\notin{[m,n]}},\jo_{[m,n]})$ equals the {\it product} of $(\cu_{w\ge n+1},\jo_1)$ and  $(\cu_{w\le m-1},\jo_2)$; that is, $\jo_{[m,n]}=\jo_1\circ \jo_2$ (we take all possible pairwise compositions of this form), and here we use the description of $\cu_{w\notin{[m,n]}}$ provided by condition \ref{iwf6n} in Theorem \ref{tww}(I).
  \end{rema}

\subsection{Examples (including stable homotopy ones)}\label{sex}

Recall that several general statements on the existence of adjacent weight and $t$-structures are given by (Theorems 2.3.4 and 2.4.2 of)  \cite{bpws} and (Theorems  3.2.3 and 4.1.2, and Proposition 4.2.1 of) \cite{bvtr}. Here we will only discuss the most "explicit" family of examples.  

Throughout this section we assume that all coproducts are small and  $\cu$ is {\it smashing}, that is, closed with respect to (small) coproducts. Note that any $\cu$ that satisfies this condition is well-known to be idempotent complete 
 (cf. Theorem \ref{tww}(II)); see Proposition 1.6.8 of \cite{neebook}. 

\begin{defi}\label{dcomp}
Assume that $\cp$ is a full subcategory of $\cu$.
\begin{enumerate}
\item\label{idcomp}
An object $M$ of $\cu$ is said to be {\it compact} if 
 the functor $H^M=\cu(M,-):\cu\to \ab$ respects coproducts. 
 
 \item\label{idcompg}
 We will say that $\cu$ is {\it compactly generated} by 
$\cp$ if $\cp$  is small,  objects of $\cp$ are compact in $\cu$, and $\cu$ equals its own  smallest strict triangulated subcategory that is closed with respect to $\cu$-coproducts and contains $\cp$.

\item\label{idneg}
 We will say that 
 $\cp$  is {\it connective} in $\cu$ if $ \cp\perp (\cup_{i>0} \cp[i])$.\footnote{ In earlier texts of the author connective subcategories were called {\it negative} ones; another related notion is {\it silting}.}
 \end{enumerate}
\end{defi}

\begin{theo}\label{thegcomp}
I. Let $\cp$ be a connective 
 subcategory of $\cu$ that compactly generates it.  
Then the following statements are valid. 

\begin{enumerate}
\item\label{itnpw} 
 Set 
   $\cu_{w\le 0}$ (resp. $\cu_{w\ge 0}$) to be the smallest subclass of $\obj \cu$ that is closed with respect to coproducts, extensions, and contains $\obj\cp[i]$ for $i\le 0$ (resp. for $i\ge 0$). Then $w=(\cu_{w\le 0}, \cu_{w\ge 0})$ is a weight structure.

Moreover, $\cu_{w\ge 0}=(\cup_{i<0}\obj\cp[i])^{\perp}$.

\item\label{itnpt} There exists a $t$-structure adjacent to $w$; respectively, $\cu_{t\ge 0}=\cu_{w\ge 0}$ and $\cu_{t\le 0}=(\cup_{i>0}\obj\cp[i])^{\perp}$.

\item\label{itnpww} 
If $m\le n$ then the corresponding class $\cu_{w\notin{[m,n]}}$ is the smallest strict class of objects of $\cu$ that is closed with respect to coproducts, extensions, and contains 
the class $$C'=\cup_{i<m\text{\ and\ }i>n}\obj\cp[i].$$

Consequently, the torsion theory $(\cu_{w\notin{[m,n]}},\cu[m,n])$ is {\it compactly generated} by 
 $C'$ in the sense of \cite[Definition 3.1.1(2)]{bpws}, that is, $\cu[m,n]=C'{}\perpp$.



\end{enumerate}

II. Take $\cu=\shtop$ (the stable homotopy category) and $\cp=\{S^0\}$ (where $S^0$ is the sphere spectrum).

\begin{enumerate}

\item\label{itnpsh} 
Then the assumptions of part I are fulfilled.

 Moreover, in this case $w$ was denoted by $\wsp$ in Theorem 4.2.4 of \cite{bkwn}; $SH_{\wsp\ge n+1}$ is the class of $n$-connected spectra, and $SH_{\wsp\le m-1}$ consists of $m-1$-skeleta in the sense of \cite[\S6.3]{marg}.
 
 \item\label{itnpsht} The class $SH_{t\le n}$  (resp. $SH_{t\ge m}=SH_{\wsp\ge m}$) is characterized by the vanishing of the stable homotopy groups $\pi_i(-)=SH(S^0,-)$  for $i>n$ (resp. for $i<m$). Consequently, $SH[0,0]$ consists of Eilenberg-Maclane spectra.  

\item\label{itnpshww} $SH_{\wsp\notin{[m,n]}}$ consists of those spectra $X$ that possess cellular towers with  $X^{(m-1)}\cong X^{(n)}$ (these are the corresponding skeleta of $X$; see  \cite[\S6.3]{marg} once again). 
Moreover, it is characterized by the vanishing of the singular homology $\hsing_{i}(X,\z)$  for $m\le i\le n$ and the freeness of  $\hsing_{m-1}(X,\z)$ (as an abelian group).

\end{enumerate}

\end{theo}
\begin{proof}


I. Assertions I.\ref{itnpw} and I.\ref{itnpt} are contained in Theorem 4.5.2 of \cite{bws}; see Theorem 3.2.3(3) and  Proposition 4.3.3(1) of \cite{bwcp} for more detail. 

I.\ref{itnpww}. Theorem \ref{tww} (see  condition I.\ref{iwf6n}  in it) implies that the class $C$ specified in the assertion contains $\cu_{w\notin{[m,n]}}$.

Conversely, $C\subset \cu_{w\notin{[m,n]}}$  since the latter class clearly contains $\obj\cp[i]$ both for all  $i<m$ and for $i>n$, and Theorem \ref{tptp}(1) immediately implies that it is closed with respect to coproducts and extensions. 

The second part of the assertion follows from our definitions easily.

II.\ref{itnpsh}. The first part of the assertion is a particular case of Theorem 4.1.1(1) of \cite{bwcp} (cf. Remark \ref{rsh}(2) below) that is well-known.
Next, the description of  $\wsp$ in the assertion is contained in Proposition 4.2.1(3) and Proposition 4.2.5 of \cite{bkwn}.

Assertion II.\ref{itnpsht} immediately follows from assertion I.\ref{itnpt}.

II.\ref{itnpshww}.  The first description of $SH_{\wsp\notin{[m,n]}}$  in the assertion 
 is an easy combination of  Theorem \ref{tww}(II) with loc. cit. The second one is provided by Theorem 4.2.4(5) of ibid.

\end{proof}

\begin{rema}\label{rsh}
1. Combining Theorem \ref{thegcomp}(\ref{itnpsh}) 
with Theorem \ref{tptp}(1) we obtain rather curious "decompositions" of objects of $SH$. Moreover, these decompositions of spectra are weakly functorial in the sense of Remark \ref{rwd}(2).

The latter theorem also gives 
 some more descriptions of morphisms killing weights and spectra without weights in a range. 
Combining some of them with Theorem \ref{thegcomp}(II.\ref{itnpsht}) we obtain the following remarkable statement: for  an $SH$-morphism $g$ we have $\hsingc^0(g,\Gamma)=0$ for any abelian group $\Gamma$ (that is, $g$ acts trivially on singular cohomology with coefficients in $\Gamma$) if and only if $g$ factors through a spectrum $X$ such that all $\hsingc^0(X,\Gamma)$ vanish. Moreover, one can also describe $SH_{\wsp\notin{[0,0]}}$ by means of  Theorem \ref{thegcomp}(II.\ref{itnpshww}) here.

2. Theorem \ref{thegcomp}(\ref{itnpsh})  can be extended to the equivariant stable homotopy $SH(G)$, where $G$ is a compact Lie group; see \S4.2 of \cite{bkwn} and \S4.1 of \cite{bwcp}. 
 Recall that the corresponding analogue of the functor $\hsingc^0(-,\Gamma)$ (see the previous part of this assertion) is the Bredon cohomology represented by an Eilenberg-MacLane $G$-spectrum; see loc. cit.

 One of the main distinctions of this general case from the case where $G=\{e\}$ (and $SH(G)=SH$) is that some of the definitions and results of \cite{marg} were never 
extended to the equivariant context.

3. The compactness assumption in 
 Theorem \ref{thegcomp}(I) can be weakened; see 
  Corollaries 2.3.1(1) and 4.1.4(II.1) of \cite{bsnew} and Remark 3.2.4(3) of \cite{bwcp}. 
\end{rema}


\end{document}